\documentclass[12pt,reqno]{amsart}

\usepackage[arrow,matrix,curve]{xy}
\usepackage{hyperref}

\usepackage{amssymb, latexsym, amsmath, amscd, array, amsthm,
setspace    
}

\usepackage{mathrsfs}
\usepackage{mparhack}
\usepackage{marginnote}
\usepackage{esvect}

\newcommand\R {{\mathbb R}} 
\newcommand\N {{\mathbb N}}

\newtheorem{theorem}{Theorem}[section]

\newtheorem{proposition}[theorem]{Proposition}

\theoremstyle{definition}
\newtheorem{definition}[theorem]{Definition}

\title[A positive function with vanishing Lebesgue integral in ZF]{A
positive function with vanishing Lebesgue integral in
Zermelo--Fraenkel set theory}

\author[V. K.]{Vladimir Kanovei} \address{V. Kanovei, Laboratory 6,
IPPI RAN, Bolshoy Karetny per.\;19, build.\;1, Moscow 127051, Russia;
and MIIT, 9b9 Obrazcova St, 127994 Moscow, Russia}
\email{kanovei@googlemail.com}

\author[M. Katz]{Mikhail G. Katz}\address{M. Katz, Department of
Mathematics, Bar Ilan University, Ramat Gan 5290002
Israel}\email{katzmik@macs.biu.ac.il}

\begin{document}

\begin{abstract}
Can a positive function on~$\R$ have zero Lebesgue integral?  It
depends on how much choice one has.  

\end{abstract}

\maketitle

\tableofcontents

\section{Introduction}

It is known that the proposition
\begin{quote}
$(\ast)$ the Lebesgue measure is countably additive
\end{quote}
is a weak form of the axiom of choice, i.e., it cannot be proved in
Zermelo--Fraenkel set theory ZF and it does not imply the axiom of
choice AC.  The Feferman--Levy model (model $\mathcal{M}9$ in the
monograph~\cite{HR}) witnesses the failure of $(\ast)$.  Thus, if one
wishes to work in a ZF environment where AC fails, one has to give up
countable additivity.  We show that, relative to a definition of
Lebesgue measure suitable for working in $\mathcal{M}9$, the Fubini
theorem and the assertion that ``a positive real-valued function on
$[0,1]$ has positive integral'' both fail in $\mathcal{M}9$.

Based on the typical working mathematician's intuitions of~$\R$, it
seems reasonable to expect that, on the one hand,
\begin{enumerate}
\item
\label{i1}
a function which is positive everywhere could not have a vanishing
Lebesgue integral; and, on the other,
\item
\label{i2}
this fact should admit a reasonably constructive proof, such as one
based on the traditional Zermelo-Fraenkel set theory (ZF).
\end{enumerate}
Curiously, while item~\eqref{i1} is correct as stated, item~\eqref{i2}
is incorrect and in fact there exist models of ZF admitting positive
functions with vanishing Lebesgue integral.

We start with the following geometric observation.  In a standard
M\"obius band~$M$, every interval orthogonal to the central circle of
$M$ is isometric to the segment~$(-\epsilon,\epsilon)$, but it is
impossible to choose such isometries simultaneously for all such
intervals on~$M$ in a continuous fashion, for that would imply
that~$M$ is a cylinder (rather than a M\"obius band).

To explain a similar phenomenon in a set-theoretic context, we first
introduce a distinction between a countable set~$A$, for which
\emph{there exists} a one-to-one correspondence between~$A$ and~$\N$,
and a \emph{counted} set~$B$, for which a specific identification
$\iota_B\colon B\to\N$ has been chosen; thus~$B$ is shorthand for the
pair~$(B,\iota_B)$.

A countable family~$\{B_{\alpha}\}_{\alpha\in \N}$ of (disjoint)
\emph{counted} sets~$B_\alpha$ is immediately identifiable with
$\N\times \N$.  Therefore the union~$\cup_{\alpha\in\N} B_\alpha$ is
countable by a snaking argument in~$\N\times \N$ familiar to each
freshman, in the context of enumerating $\mathbb Q$.

The snaking argument does not \emph{a priori} work in ZF for a
countable family~$\{A_{\alpha}\}_{\alpha\in \N}$ of (disjoint)
\emph{countable} sets~$A_\alpha$, because constructing an
identification with~$\N\times\N$ requires \emph{choosing} a particular
identification of~$A_\alpha$ with~$\N$ \emph{simultaneously} for
all~$\alpha$.  This procedure requires the axiom of countable choice
(ACC).

This might seem like a limitation of the particular proof of
countability that one might be able to overcome by a more judicious
procedure avoiding ACC.  However, it turns out that the obstruction is
genuine.

Cohen's work on the continuum hypothesis soon led to a model of~ZF,
called the Feferman--Levy (FL) model \cite{FL}, where the real number
line~$\R$ is a countable disjoint union of countable (but \emph{not}
counted) sets, whereas of course~$\R$ itself is uncountable by the
classical diagonal argument of Cantor (which does not rely on AC); see
\cite[chapter IV, section 4]{coh}.  The Lebesgue measure is
not~$\sigma$-additive in FL.
 
These phenomena do not contradict the categoricity of~$\R$ (namely,
its characterisation as the unique complete ordered field) since, as
is well known, the concept of categoricity is dependent on the
background model of set theory.

\section{On Lebesgue measure in the Feferman--Levy model}
\label{s2}

A crucial property of the Feferman--Levy model is the following:

\begin{quote}
(FL) 
The unit interval $[0,1]$ is equal to a countable union 
$[0,1]=\bigcup_{n} C_n$ 
of countable sets $C_n\subseteq [0,1]$.
\end{quote}
Due to property (FL), not every definition of the Lebesgue measure is
appropriate in this model.  In particular, any definition that
explicitly involves the algebra of Borel sets and countable
additivity, as e.g., in Halmos \cite{Ha76}, is immediately ruled out
as incompatible with property (FL).%
\footnote {However see Fremlin \cite{Fr08} for a hard roundabout in
terms of \emph{codable} countable unions, which enables one to salvage
\emph{some} countable additivity even in fully non-AC environments.}
Therefore we begin with an explicit definition of the Lebesgue measure
suitable for working with in any ZF environment, including the
choiceless FL model.  A set $X\subseteq \R$ is \emph{bounded} if
$X\subseteq[-c,c]$ for some finite $c<+\infty$.  We will consider the
measure with respect to bounded sets only.  The \emph{length}
$\text{len}(I)$ of any open or closed real interval $(a,b)$ or $[a,b]$
is equal to $b-a$.

\begin{definition}
[see \cite{Ta11}, 1.2.2, or \cite{KF} in slightly different terms]
The outer measure $m^*(X)$ of a bounded set $X\subseteq \R$ is equal
to the infimum of $\sum_n {\rm len}(I_n)$ over all coverings
$X\subseteq \bigcup_n I_n$ of $X$ by countably many open
intervals~$I_n$.

A bounded set $X\subseteq\R$ is \emph{Lebesgue measurable} (LM for
short) if for every $\varepsilon > 0$, there exists an open set $U$
containing $X$ such that $m^*(U\setminus X)<\varepsilon$.  If $X$ is
Lebesgue measurable, then $m(X) := m^*(X)$ is the Lebesgue measure
of~$X$.

The same definition of $m^*$ and $m$ works in any $n$-dimensional
Euclidean space $\R^n$, $n\ge2$, with intervals replaced by
\emph{boxes}, i.e., Cartesian products of $n$-many finite real
intervals.
\end{definition}
The following properties of the measure are provable in ZF and hence
true in the FL model.
\begin{enumerate}
\item[(I)] every interval $I=[a,b]$ or $(a,b)$ is LM and satisfies
$m(I)=\text{len}(I)$;
\item[(II)] $m$ is finitely additive;
\item[(III)] every finite or countable bounded set $X\subseteq \R$ is
LM and satisfies $m(X)=0$.
\end{enumerate}
It follows by (FL) that $m$ is {\it not\/} countably additive in 
the FL model. 

\begin{definition}
\label{d22}
A real function $f\colon I\to \R$ defined on a bounded real interval
$I$ is {\it Lebesgue measurable\/} (LM) iff each superlevel set
$\{x\in I\colon f(x) > y\}$ is LM.  In such case the auxiliary
function $g(y)=m(\{x\in I\colon f(x) > y\})$ is monotone and has a
well-defined Riemann integral.  The Lebesgue integral $\int_I f dx$ is
defined to be equal to the Riemann integral $\int_0^\infty g(y) dy$.
\end{definition}

\section{A positive function with zero Lebesgue integral in FL}

D.\;Fremlin \cite{Fr08} presents a careful development of Lebesgue
integration in a ZF context and shows that it satisfies a fundamental
theorem of calculus.  Note that various definitions that are
equivalent in a ZFC context may become inequivalent over ZF.  

We show that using the traditional definitions presented in
Section~\ref{s2} the Lebesgue integral of our positive function is
zero.

\begin{theorem}
[in ZF] 
\label{t11}
Suppose that~$[0,1]=\bigcup_{n\in\N} A_n$, where each~$A_n$ is
countable (as e.g., in FL) .  Then there is a positive real function
on~$[0,1]$ with zero Lebesgue integral.
\end{theorem}

\begin{proof}
Suppose~$[0,1]=\bigcup_{n\in\N} A_n$ where each~$A_n$ is countable.
We will rely on Definition~\ref{d22} of the Lebesgue integral.
Consider the function~$f$ equal to~$\frac{1}{n}$ on the~$n$-th
countable set~$A_n, n=1,2,3,\ldots$, or in
formulas~$f\hspace{-5pt}\downharpoonright_{A_n}^{\phantom{I}}=\frac{1}{n}$;
alternatively, $f=\sum_n \frac{1}{n} \chi_{A_n}^{\phantom{I}}$.  Then
the auxiliary function~$g$ is identically zero.  Indeed, its
superlevel sets are finite unions of the form
$A_1\cup{}A_2\cup\dots\cup A_n$, hence, countable sets.%
\footnote {Note that there is no need for the axiom of choice to prove
that a \emph{finite} union of countable sets is itself countable, and
that a countable set has Lebesgue measure~$0$; see \cite{Fr08}.
Furthermore, $f$ is a Lebesgue-measurable function; indeed the
$f$-preimage of any interval $(a,+\infty)$ is either the entire
interval $[0,1]$ if $a\le 0$, or a finite union of sets $A_n$
otherwise, in which case it is a countable set.}
Therefore the Riemann integral of $g$ is zero, and hence the Lebesgue
integral of the given positive function $f$ is also zero.

The definition of Lebesgue integral in terms of simple functions gives
the same result.  For our function~$f(x)$ which equals $\frac{1}{n}$
when~$x\in A_n$, every simple function $s(x)\geq 0$ dominated by~$f$
will be nonzero at most at countably many points.  Therefore the
Riemann integral of $s(x)$ is zero.  Therefore the Lebesgue integral
of~$f$, which is the supremum of the integrals of the simple
functions, is also zero.  The infimum of the integrals of simple
functions that dominate~$f$ is similarly~$0$.  

Indeed, given a small $\epsilon=\frac1m$ where $m\ge 1$, we define $s$
by setting $s(x)=f(x)=\frac1n$ whenever $x$ belongs to $A_n$ and $n\le
m$, and $s(x)=\frac1m$ whenever $x$ belongs to $A_n$ and $n>m$.
Then~$s$ is a simple function whose Lebesgue integral by definition is
equal to $I_1+I_2+\dots+I_m+I$, where each $I_n= \frac1n\cdot\mu(A_n)$
is equal to $0$ since $A_n$ is countable, while
$I=\frac1m\mu(\bigcup_{n>m}A_n)\le\frac1m\mu([0,1])=\frac1m$, so
altogether~$\int_0^1s\,d\mu\le\frac1m$, as required.  Here $\mu$
denotes the Lebesgue measure on $[0,1]$, and the axiom of choice is
not used in this elementary argument.
\end{proof}

\section
{Violation of Fubini}

The original and best known form of the Fubini Theorem claims that
under certain conditions two iterated integrals coincide. However
there is another formulation of the result, which does not directly
involve integration. This form, useful in the context of Lebesgue
measure theory, is similar to the Kuratowski--Ulam theorem (see
\cite[8.41]{Ke95}) related to Baire category rather than measure.
Namely we have the following; cf.~\cite[exercise 565X(e),
p.\;221]{Fr08}.

\begin{proposition}[Fubini]
\label{p31}
Assume that~$P\subseteq \mathbb R^2$ is a Lebesgue measurable set.
Then~$P$ is null if and only if the set~$X$ of all reals~$x$ such that
``the cross-section~$P_x=\{y\colon(x,y)\in P\}$ is non-null'', is null
as well.
\end{proposition}

\begin{theorem}[in ZF]
In the hypotheses of Theorem\;$\ref{t11}$, Proposition\;$\ref{p31}$
fails.
\end{theorem}

\begin{proof}
Consider a decomposition $I=\bigcup_nA_n$ of the unit
interval~$I=[0,1]$ where each~$A_n$ is countable.  Now let~$f\colon
I\to\mathbb R$ be the function defined as in the proof of
Theorem~\ref{t11}.  Let~$P$ be the set of all pairs~$(x,y)$ such that
$0<y<f(x)$. We claim that~$P$ is a null set. Indeed, given
$\epsilon=\frac1n$, we have to cover~$P$ by an open planar set~$U$ of
measure~$<\frac2n$.  We let~$P=P'\cup P''$, where~$P'=\{(x,y)\in
P\colon x \in \bigcup_{k<n}A_k\}$ and~$P''=\{(x,y)\in P\colon x \in
\bigcup_{k\ge n}A_k\}$.  Note that if~$x\in A_k$,~$k\ge n$,
then~$f(x)\le \frac1n$, therefore~$P''$ is covered by the open
rectangle~$U''=(0,1)\times (0,\frac1n)$ of measure~$\frac1n$.

On the other hand, the projection of~$P'$ to the horizontal axis is
equal to the countable set~$I'=\bigcup_{k<n}A_k$. Therefore~$I'$ is a
null set, and we can cover it by an open set~$G$ of measure
$<\frac1n$.  Then~$P'$ is covered by~$U'=G\times[0,1]$, an open set
of planar measure~$<\frac1n$.

Finally, the union~$U=U'\cup U''$ is an open planar set of measure
$<\frac2n$ which covers~$P$, as required.  Thus~$P$ is a null set.

On the other hand, if~$x\in I$ then the vertical cross-section~$P_x$
defined by~$P_x=\{y\colon(x,y)\in P\}$ is a non-empty open interval,
hence, a non-null set. It follows that the set~$X$ as in Proposition
\ref{p31} is equal to~$I$, and hence is non-null. Thus~$P$ violates
Proposition~\ref{p31}, as required.
\end{proof}

There are various definitions of integration, and the relations among
them and the Lebesgue measure, to which mathematicians are accustomed,
sometimes depend on the axiom of choice (AC).  Accordingly, if~AC
fails then such relationships may fail as well.  In particular, in
the~FL model, where even countable choice is not available, an
elementary standard property of integration fails as well, as we
demonstrate in this note.

\section*{Acknowledgments}

We are grateful to the referees for a number of helpful suggestions.
V.\;Kanovei was supported in part by the RFBR grant
number~17-01-00705.  M.\;Katz was partially funded by the Israel
Science Foundation grant number~1517/12.

\end{document}